\documentclass[a4paper,reqno,11pt]{amsart}

\usepackage{amsmath}
\usepackage{amssymb}
\usepackage{amsthm}
\usepackage{amsfonts}

\usepackage[T1]{fontenc}

\usepackage{mathtools}
\usepackage{enumerate}
\usepackage[hidelinks]{hyperref}
\usepackage[nameinlink]{cleveref}

\usepackage[backrefs]{amsrefs}

\usepackage{color}

\usepackage{xpatch}
\xpatchcmd{\proof}{\itshape}{\bfseries}{}{}

\numberwithin{equation}{section}

\theoremstyle{plain}
\newtheorem{theorem}{Theorem}
\newtheorem{proposition}{Proposition}[section]
\newtheorem{lemma}[proposition]{Lemma}
\newtheorem{corollary}[proposition]{Corollary}

\theoremstyle{definition}
\newtheorem{definition}{Definition}[section]

\theoremstyle{remark}
\newtheorem{remark}[proposition]{Remark}

\makeatletter
\@addtoreset{step}{section}
\makeatother

\newcommand{\R}{\mathbb{R}}

\newcommand{\cH}{\mathcal{H}}
\newcommand{\cM}{\mathcal{M}}
\newcommand{\cN}{\mathcal{N}}

\makeatletter
\newcommand{\oset}[3][0ex]{\mathrel{\mathop{#3}\limits^{\vbox to#1{\kern-2\ex@\hbox{\(\scriptstyle#2\)}\vss}}}}
\makeatother

\newcommand\cl[1]{\overline{#1}}
\newcommand{\loc}{\mathrm{loc}}
\newcommand{\weakstartext}{weak\({}^*\)}
\newcommand{\weakstarto}{\oset{\ast}{\rightharpoonup}}
\newcommand{\weakto}{\rightharpoonup}

\newcommand{\defeq}{\vcentcolon=}

\renewcommand{\d}{\mathop{}\!\mathrm{d}}

\DeclarePairedDelimiter{\abs}{\lvert}{\rvert}

\DeclarePairedDelimiter{\norm}{\lVert}{\rVert}
\DeclarePairedDelimiter{\paren}{\lparen}{\rparen}

\DeclarePairedDelimiter{\set}{\lbrace}{\rbrace}

\DeclareMathOperator{\sgn}{sgn}

\title{The Hopf lemma for the Schr\"odinger operator}
\author{Augusto C. Ponce}
\address{
Augusto C. Ponce\hfill\break\indent
Universit{\'e} catholique de Louvain\hfill\break\indent
Institut de Recherche en Math{\'e}matique et Physique\hfill\break\indent
Chemin du Cyclotron 2, bte L7.01.02\hfill\break\indent
1348 Louvain-la-Neuve\hfill\break\indent
Belgium}
\email{Augusto.Ponce@uclouvain.be}

\author{Nicolas Wilmet}
\address{
Nicolas Wilmet\hfill\break\indent
Universit{\'e} catholique de Louvain\hfill\break\indent
Institut de Recherche en Math{\'e}matique et Physique\hfill\break\indent
Chemin du Cyclotron 2, bte L7.01.02\hfill\break\indent
1348 Louvain-la-Neuve\hfill\break\indent
Belgium}
\email{Nicolas.Wilmet@uclouvain.be}



\begin{document}

\maketitle

\begin{abstract}
\noindent
We prove the Hopf boundary point lemma for solutions of the Dirichlet problem involving the Schr\"odinger operator \(- \Delta + V\) with a nonnegative potential \(V\) which merely belongs to \(L_\loc^1(\Omega)\). 
More precisely, if \(u \in W_0^{1, 2}(\Omega) \cap L^2(\Omega; V \d{x})\) satisfies \(- \Delta u + V u = f\) on \(\Omega\) for some nonnegative datum \(f \in L^\infty(\Omega)\), \(f \not\equiv 0\), then we show that at every point \(a \in \partial\Omega\) where the classical normal derivative \(\partial u(a) / \partial n\) exists and satisfies the Poisson representation formula, one has \(\partial u(a) / \partial n > 0\) if and only if the boundary value problem
\[
\begin{dcases}
\begin{aligned}
- \Delta v + V v &= 0 && \text{in \(\Omega\),} \\
v &= \nu && \text{on \(\partial\Omega\),}
\end{aligned}
\end{dcases}
\]
involving the Dirac measure \(\nu = \delta_a\) has a solution. More generally, we characterize the nonnegative finite Borel measures \(\nu\) on \(\partial\Omega\) for which the boundary value problem above has a solution in terms of the set where the Hopf lemma fails.
\end{abstract}


\section{Introduction and main results}

Let \(\Omega\) be a bounded connected open subset of \(\R^N\) with smooth boundary and let \(V \in L^\infty(\Omega)\) be a nonnegative function. The weak maximum principle ensures that the distributional solution \(u \in C^1(\cl\Omega)\) of the Dirichlet problem
\begin{equation}
\label{eq:dirichlet_problem}
\begin{dcases}
\begin{aligned}
- \Delta u + V u &= f && \text{in \(\Omega\),} \\
u &= 0 && \text{on \(\partial\Omega\),}
\end{aligned}
\end{dcases}
\end{equation}
satisfies \(u \ge 0\) on \(\Omega\) whenever \(f \in L^\infty(\Omega)\) is a nonnegative function; see \Cref{lem:weak_maximum_principle} below. From the minimality of \(u\) on \(\partial\Omega\), the normal derivative of \(u\) with respect to the inward unit normal vector \(n\) thus verifies \(\partial u / \partial n \ge 0\) on \(\partial\Omega\). When \(f \not\equiv 0\), the classical Hopf lemma (see \cite{Evans:2010}*{Lemma~6.4.2} or \cite{GilbargTrudinger:2001}*{Lemma~3.4}) gives the stronger conclusion
\begin{equation}
\label{eq:Hopf}
\frac{\partial u}{\partial n} > 0 \quad \text{on \(\partial\Omega\).}
\end{equation}

Boundedness of \(V\) is an important element to obtain \eqref{eq:Hopf} as it allows one to construct a positive minorant of \(u\) on \(\Omega\) with positive normal derivative at any given point on \(\partial\Omega\).{}
To understand in what respect this assumption on \(V\) can be relaxed, we assume henceforth that
\[
\boxed{
\text{\(V \in L_\loc^1(\Omega)\) and \(V \ge 0\) almost everywhere on \(\Omega\),}
}
\]
but we restrict ourselves to the class of nonnegative data \(f \in L^\infty(\Omega)\). 
In this setting, a solution of \eqref{eq:dirichlet_problem} is a function \(u\) that belongs to \(W_0^{1, 2}(\Omega) \cap L^2(\Omega; V \d{x})\) and satisfies the equation
\[
- \Delta u + V u = f \quad \text{in the sense of distributions in \(\Omega\).}
\]
Observe that \(u\) is the unique minimizer of the energy functional
\[
E(z) = \frac{1}{2} \int_\Omega {} (\abs{\nabla z}^2 + V z^2) \d{x} - \int_\Omega f z \d{x}
\]
with \(z \in W_0^{1, 2}(\Omega) \cap L^2(\Omega; V \d{x})\).

As the solution of \eqref{eq:dirichlet_problem} need not be \(C^1\), nor even continuous, due to some possible singularity from \(V\), we first need to address the pointwise meaning of the normal derivative \(\partial u / \partial n\). 
Since \(u\) is the difference between a continuous and a bounded superharmonic function, every \(x \in \Omega\) is a Lebesgue point and the precise representative of \(u\) satisfies the following representation formula in terms of the Green function \(G\) of \(- \Delta\) on \(\Omega\)~:
\[
\hat{u}(x) = \int_\Omega G(x, y) (- \Delta u(y)) \d{y} \quad \text{for every \(x \in \Omega\).}
\]
Then, from a formal computation, one presumably gets at a point \(a \in \partial\Omega\)~:
\begin{equation}
\label{eq:poisson_integral}
\frac{\partial \hat{u}}{\partial n}(a) = \int_\Omega K(a, y) (- \Delta u(y)) \d{y},
\end{equation}
where \(K \defeq \partial G / \partial n\) denotes the Poisson kernel of \(- \Delta\) on \(\Omega\). This formula can be rigorously justified when \(V \in L^\infty(\Omega)\), and then \(\Delta u \in L^\infty(\Omega)\), using standard estimates on \(G\).

There is no reason why \eqref{eq:poisson_integral} should remain valid in general as we do not assume any particular behaviour of \(V\) near \(\partial\Omega\). 
We show nevertheless that, for any fixed \(V\), there is a common property which is shared by all nontrivial solutions of \eqref{eq:dirichlet_problem} with nonnegative \(f \in L^\infty(\Omega)\). To this end, let \(\zeta_1\) be the solution of \eqref{eq:dirichlet_problem} with constant density \(f \equiv 1\) and define the set
\[{}
\cN = 
\left\{
a \in \partial\Omega{}
\left|{}
\begin{aligned}
\;  & \text{the classical normal derivative \(\partial\widehat{\zeta_{1}}/\partial n\) exists at \(a\)}\\
	& \text{and \eqref{eq:poisson_integral} is valid with \(u = \zeta_{1}\)}
\end{aligned}
\right.
\right\}.
\]
To simplify the notation, we do not explicit the dependence of \(\cN\) on \(V\).

We prove

\begin{theorem}
\label{thm:universal}
For every nonnegative function \(f \in L^\infty(\Omega)\), \(f \not\equiv 0\), the solution \(u\) of \eqref{eq:dirichlet_problem} involving \(f\) has a classical normal derivative at \(a \in \partial\Omega\) that satisfies \eqref{eq:poisson_integral} if and only if \(a \in \cN\).
\end{theorem}

The set \(\cN\) thus provides one with a common ground where a normal derivative exists, independently of the solution of \eqref{eq:dirichlet_problem}. We can now address the question of whether the Hopf lemma is valid on \(\cN\). We rely on the characterization of the set of points \(a \in \partial\Omega\) for which the boundary value problem
\begin{equation}
\label{eq:boundary_value_problem_dirac}
\begin{dcases}
\begin{aligned}
- \Delta v + V v &= 0 && \text{in \(\Omega\),} \\
v &= \delta_a && \text{on \(\partial\Omega\),}
\end{aligned}
\end{dcases}
\end{equation}
involving the Dirac measure \(\delta_a\) has a distributional solution in the sense that \(v \in L^1(\Omega)\) is such that \(V v \in L^1(\Omega; d_{\partial\Omega} \d{x})\) and satisfies
\begin{equation}
\label{eq:representation_smooth_functions}
\frac{\partial \zeta}{\partial n}(a) = \int_\Omega v (- \Delta \zeta + V \zeta) \d{x}
\end{equation}
for every \(\zeta \in C^\infty(\cl\Omega)\) with \(\zeta = 0\) on \(\partial\Omega\), where \(d_{\partial\Omega} : \Omega \to \R_+\) is the distance to the boundary. 
When the test function \(\zeta\) is non-identically zero and satisfies \(- \Delta \zeta + V \zeta \ge 0\) on \(\Omega\), it follows from \eqref{eq:representation_smooth_functions} and the strong maximum principle for the Schr\"odinger operator with potential in \(L^{1}_{\loc}\) (see \cite{Ancona:1979}*{Th\'eor\`eme~9}, \cite{BrezisPonce:2003}*{Theorem~1} or \cite{Trudinger:1978}) that \(\partial \zeta(a) / \partial n > 0\). It is therefore reasonable to expect the validity of the Hopf lemma for \eqref{eq:dirichlet_problem} on the set of points \(a \in \partial\Omega\) for which the boundary value problem \eqref{eq:boundary_value_problem_dirac} has a solution. This motivates the following

\begin{definition}
\label{def:exceptional_boundary_set}
The \emph{exceptional boundary set} \(\Sigma\) associated to \(- \Delta + V\) is the set of points \(a \in \partial\Omega\) for which the boundary value problem \eqref{eq:boundary_value_problem_dirac} with datum \(\delta_a\) does not have a distributional solution.
\end{definition}

We can now state the Hopf lemma on \(\cN\)\,:

\begin{theorem}
\label{thm:hopf}
Let \(u\) be the solution of \eqref{eq:dirichlet_problem} for some nonnegative datum \(f \in L^\infty(\Omega)\), \(f\not\equiv 0\). For every \(a \in \cN\), we have
\[
\frac{\partial \hat{u}}{\partial n}(a) > 0 \quad \text{if and only if} \quad a \not\in \Sigma.
\]
\end{theorem}

In the case where \(V \in L^q(\Omega)\) for some \(q > N\), one has \(\cN = \partial\Omega\) and \(\Sigma = \emptyset\). 
Hence,
\[
\frac{\partial \hat{u}}{\partial n}(a) > 0 \quad \text{for every \(a \in \partial\Omega\)\,;}
\]
see \Cref{cor:hopf_lemma_supercritical} below. As one can expect, the validity of the Hopf lemma depends on the behaviour of \(V\) near the boundary. For example, under the assumption that
\begin{equation}
\label{eq:subquadratic}
V \le C / d_{\partial\Omega}^2 \quad \text{almost everywhere on \(\Omega\)}
\end{equation}
for some constant \(C \ge 0\), Ancona established in \cite{Ancona:2012} (see also the Appendix in \cite{VeronYarur:2012}) a beautiful characterization of the set of points where \eqref{eq:boundary_value_problem_dirac} has a solution: \(a \in \partial\Omega \setminus \Sigma\) if and only if the Poisson kernel of \(- \Delta\) at \(a\) is a supersolution of \eqref{eq:dirichlet_problem}. 
Using his result and the pointwise behaviour of \(K\), we can state the following

\begin{corollary}
Assume that \(V\) satisfies \eqref{eq:subquadratic} and let \(u\) be the solution of \eqref{eq:dirichlet_problem} for some nonnegative datum \(f \in L^\infty(\Omega)\), \(f \not\equiv 0\). For every \(a \in \cN\), we have
\[
\frac{\partial \hat{u}}{\partial n}(a) > 0 \quad \text{if and only if} \quad \int_\Omega \frac{d_{\partial\Omega}^2(y)}{\abs{y - a}^N} V(y) \d{y} < + \infty.
\]
\end{corollary}

Quadratic blow-up of the potential as in \eqref{eq:subquadratic} is a threshold for the validity of the Hopf lemma. More precisely,

\begin{corollary}
\label{cor:superquadratic}
Assume that \(V\) satisfies
\[
V \ge C' / d_{\partial\Omega}^2 \quad \text{almost everywhere on \(\Omega\)}
\]
for some \(C' > 0\). Then \(\cN = \partial\Omega\) and, for every solution \(u\) of \eqref{eq:dirichlet_problem} with \(f \in L^\infty(\Omega)\), we have
\[
\frac{\partial \hat{u}}{\partial n} = 0 \quad \text{on \(\partial\Omega\).}
\]
\end{corollary}

\Cref{cor:superquadratic} is a consequence of our \Cref{thm:universal} and a result from D\'iaz~\cite{Diaz:2017} which establishes the existence of a bounded nonnegative eigenfunction \(u\) for \(- \Delta + V\) that satisfies a Dirichlet problem of the type \eqref{eq:dirichlet_problem} and such that \(\partial \hat{u}(a) / \partial n = 0\) for every \(a \in \partial\Omega\).

Although we have introduced the exceptional set \(\Sigma\) by dealing with Dirac masses on \(\partial\Omega\), the set \(\Sigma\) allows one to characterize all nonnegative finite Borel measures \(\nu\) on \(\partial\Omega\) for which the boundary value problem
\begin{equation}
\label{eq:boundary_value_problem}
\begin{dcases}
\begin{aligned}
- \Delta v + V v &= 0 && \text{in \(\Omega\),} \\
v &= \nu && \text{on \(\partial\Omega\),}
\end{aligned}
\end{dcases}
\end{equation}
has a distributional solution. This is the content of our next theorem that extends a previous result by V\'eron and Yarur~\cite{VeronYarur:2012}:

\begin{theorem}
\label{thm:characterization_good_measures}
The boundary value problem \eqref{eq:boundary_value_problem} associated to a nonnegative finite Borel measure \(\nu\) on \(\partial\Omega\) has a distributional solution if and only if \(\nu(\Sigma) = 0\).
\end{theorem}

The proof of \Cref{thm:characterization_good_measures} is inspired by the recent paper of Orsina and the first author \cite{OrsinaPonce:2019} concerning the failure of the strong maximum principle for the Schr\"odinger operator \(- \Delta + V\) in the case where \(V\) is merely a nonnegative Borel measurable function. 
In this respect, we introduce in \Cref{sec:pointwise_normal_derivative} a notion of pointwise normal derivative for solutions of \eqref{eq:dirichlet_problem} that is defined everywhere on \(\partial\Omega\), but possibly depends on the potential \(V\). 
In \Cref{sec:boundary_value_problem}, we present a counterpart for \eqref{eq:boundary_value_problem} of the notion of duality solution introduced by Malusa and Orsina \cite{MalusaOrsina:1996}. The exceptional boundary set \(\Sigma\) is then identified in \Cref{sec:normal_derivative_on_sigma} with the set of boundary points at which all such normal derivatives vanish.
Using the tools developed in \Cref{sec:boundary_value_problem,sec:normal_derivative_on_sigma}, we prove \Cref{thm:characterization_good_measures} in \Cref{sec:proof_characterization_good_measures}. 
\Cref{thm:universal,thm:hopf} are proved in \Cref{sec:universal,sec:proof_hopf}, respectively.

\section{Pointwise normal derivative associated to the Schr\"odinger operator}
\label{sec:pointwise_normal_derivative}

A property that is common to all solutions of \eqref{eq:dirichlet_problem} concerns the existence of a distributional normal derivative as an element in \(L^{1}(\partial\Omega)\).{}
This is a general feature that relies on the facts that
\[{}
u \in W_{0}^{1, 1}(\Omega){}
\quad \text{and} \quad{}
\Delta u \text{ is a finite Borel measure on \(\Omega\).}
\]
Brezis and the first author proved in \cite{BrezisPonce:2008} that in this general setting there exists a function in \(L^1(\partial\Omega)\), which is denoted by \(\partial u / \partial n\) and coincides with the classical normal derivative when \(u\) is a \(C^{2}\)~function, that satisfies
\[
\norm[\bigg]{\frac{\partial u}{\partial n}}_{L^1(\partial\Omega)} \le \abs{\Delta u}(\Omega)
\]
and
\begin{equation}
\label{eq:distributional_normal_derivative_identity}
\int_\Omega \nabla u \cdot \nabla \psi \d{x} = - \int_\Omega \psi \Delta u  - \int_{\partial\Omega} \frac{\partial u}{\partial n} \psi \d\sigma \quad \text{for every \(\psi \in C^\infty(\cl\Omega)\),}
\end{equation}
where \(\sigma = \cH^{N - 1}\lfloor_{\partial\Omega}\) is the surface measure on \(\partial\Omega\)\,; see \cite{BrezisPonce:2008}*{Theorem~1.2} or \cite{Ponce:2016}*{Proposition~7.3}. We recall that \(n\) is the inward unit normal vector, which explains the minus sign in front of the second integral in the right-hand side of \eqref{eq:distributional_normal_derivative_identity}. When \(u \ge 0\) almost everywhere on \(\Omega\), one additionally has
\begin{equation}
\label{eq:distributional_normal_derivative_nonnegativity}
\frac{\partial u}{\partial n} \ge 0 \quad \text{almost everywhere on \(\partial\Omega\)\,;}
\end{equation}
see \cite{BrezisPonce:2008}*{Corollary~6.1} or \cite{Ponce:2016}*{Lemma~12.15}. Since the mapping
\[
\set[\big]{u \in W_0^{1, 1}(\Omega) : \Delta u \in L^1(\Omega)} \to L^1(\partial\Omega) : u \mapsto \frac{\partial u}{\partial n}
\]
is linear, such a property yields a handy comparison principle: 
if \(v\) and \(w\) both satisfy \eqref{eq:dirichlet_problem}, with possibly different potentials \(V\) and data \(f\), and if \(v \le w\) almost everywhere on \(\Omega\), then
\[
\frac{\partial v}{\partial n} \le \frac{\partial w}{\partial n} \quad \text{almost everywhere on \(\partial\Omega\).}
\]

More specific to solutions of \eqref{eq:dirichlet_problem}, we show that there is a notion of \emph{pointwise} normal derivative that is adapted to the Schr\"odinger operator \(- \Delta + V\) and used in the proofs of \Cref{thm:universal,thm:hopf,thm:characterization_good_measures}. For this purpose, let \((V_k)\) be a nondecreasing sequence of nonnegative functions in \(L^\infty(\Omega)\) that converges almost everywhere to \(V\) on \(\Omega\). The construction of this pointwise normal derivative relies on the main result of this section which is

\begin{proposition}
\label{prop:pointwise_normal_derivative}
Let \(u\) be the solution of \eqref{eq:dirichlet_problem} associated to \(f \in L^\infty(\Omega)\) and denote by \(u_k\) the solution of
\begin{equation}
\label{eq:dirichle_proposition_pointwise}
\begin{dcases}
\begin{aligned}
- \Delta u_k + V_k u_k &= f && \text{in \(\Omega\),} \\
u_k &= 0 && \text{on \(\partial\Omega\).}
\end{aligned}
\end{dcases}
\end{equation}
Then
\begin{enumerate}[(i)]
\item \(u_k \to u\) and \(V_k u_k \to V u\) in \(L^1(\Omega)\) and almost everywhere on \(\Omega\)\,;
\item \(\paren{\partial u_k / \partial n}\) is uniformly bounded on \(\partial\Omega\)\,;
\item \(\paren{\partial u_k / \partial n}\) converges pointwise to a function \(g : \partial\Omega \to \R\) such that \(g = \partial u / \partial n\) almost everywhere on \(\partial\Omega\) and, for every \(N < p \le \infty\),
\[
\norm{g}_{L^\infty(\partial\Omega)} \le C \norm{f}_{L^p(\Omega)}
\]
with a constant \(C > 0\) depending on \(p\) and \(\Omega\). Moreover, \(g \ge 0\) on \(\partial\Omega\) whenever \(f \ge 0\) almost everywhere on \(\Omega\).
\end{enumerate}
\end{proposition}

Since \(V_k\) is bounded, we have \(u_k \in C^1(\cl\Omega)\) and in particular the classical normal derivative \(\partial u_k / \partial n\) is well-defined on \(\partial\Omega\). 
To see why this is true, let \(w \in C^\infty(\cl\Omega)\) be the solution of
\[
\begin{dcases}
\begin{aligned}
- \Delta w &= \abs{f} && \text{in \(\Omega\),} \\
w &= 0 && \text{on \(\partial\Omega\).}
\end{aligned}
\end{dcases}
\]
The weak maximum principle implies that \(\abs{u_k} \le w\) almost everywhere on \(\Omega\)\,; thus \(u_k \in L^\infty(\Omega)\). 
Since \(V_k\) and \(f\) are bounded, we have \(\Delta u_k \in L^\infty(\Omega)\), hence \(u_k \in W^{2, p}(\Omega)\) for every \(1 < p < \infty\)\,; see \cite{GilbargTrudinger:2001}*{Theorem~9.15 and Lemma~9.17}. 
Taking any \(p > N\), it follows from the Morrey--Sobolev embedding theorem that \(u_k \in C^1(\cl\Omega)\)\,; see \cite{Willem:2013}*{Theorem~6.4.4}. 
In addition, one has the estimate
\begin{equation}
\label{eq:C1_estimate}
\norm{w}_{C^1(\cl\Omega)} 
\le C \norm{\Delta w}_{L^p(\Omega)}
= C \norm{f}_{L^p(\Omega)}
\end{equation}
for some constant \(C > 0\) depending on \(p\) and \(\Omega\). Since
\[
\abs[\bigg]{\frac{\partial u_k}{\partial n}} \le \frac{\partial w}{\partial n} \quad \text{on \(\partial\Omega\),}
\]
one deduces from \eqref{eq:C1_estimate} that
\begin{equation}
\label{eq:Linf_boundary_estimate}
\norm[\bigg]{\frac{\partial u_k}{\partial n}}_{L^\infty(\Omega)} \le \norm{w}_{C^1(\cl\Omega)} \le C \norm{f}_{L^p(\Omega)}.
\end{equation}

Using \Cref{prop:pointwise_normal_derivative}, we then define the pointwise normal derivative of \(u\) with respect to \(- \Delta + V\) as
\[
\boxed{
\frac{\widehat{\partial u}}{\partial n}(a) \defeq g(a) \quad \text{for every \(a \in \partial\Omega\).}
}
\]
At first sight, this definition could depend on the choice of approximation \((V_k)\) like
\[
V_k = \min \set{V, k},
\]
but as we shall see later on it does not; see \Cref{rem:independence_from_Vk}. As a consequence of assertion \emph{(iii)} in \Cref{prop:pointwise_normal_derivative}, \(\widehat{\partial u} / \partial n\) is a distributional normal derivative of \(u\).

Before proceeding with the proof of \Cref{prop:pointwise_normal_derivative}, we recall standard estimates for solutions of the Dirichlet problem
\begin{equation}
\label{eq:dirichlet_problem2}
\begin{dcases}
\begin{aligned}
- \Delta u + V u &= \mu && \text{in \(\Omega\),} \\
u &= 0 && \text{on \(\partial\Omega\),}
\end{aligned}
\end{dcases}
\end{equation}
where \(\mu \in L^1(\Omega)\). By a solution of \eqref{eq:dirichlet_problem2}, we mean a function \(u \in W_0^{1, 1}(\Omega) \cap L^1(\Omega; V\d{x})\) that satisfies the equation in the sense of distributions in \(\Omega\). 
For all \(1 \le p < \frac{N}{N - 1}\), the solution exists, is unique and belongs to \(W_0^{1, p}(\Omega)\) with
\begin{equation}
\label{eq:W1p_estimate}
\norm{u}_{W^{1, p}(\Omega)} \le C \norm{\mu}_{L^1(\Omega)}
\end{equation}
for some constant \(C > 0\) depending on \(p\) and \(\Omega\).{}
This can be deduced from elliptic estimates due to Littman, Stampacchia and Weinberger \cite{LittmanStampacchiaWeinberger:1963}*{Theorem~5.1} and from the absorption estimate
\begin{equation}
\label{eq:absorption_estimate}
\norm{V u}_{L^1(\Omega)} \le \norm{\mu}_{L^1(\Omega)}.
\end{equation}
The latter inequality can be obtained using as test function a suitable approximation of the sign function \(\sgn u\)\,; see \cite{BrezisMarcusPonce:2007}*{Proposition~4.B.3} or \cite{Ponce:2016}*{Proposition~21.5}.

The weak maximum principle for \eqref{eq:dirichlet_problem2} that is mentioned in the introduction is justified by the following

\begin{lemma}
\label{lem:weak_maximum_principle}
Let \(u\) be the solution of \eqref{eq:dirichlet_problem2} involving \(\mu \in L^1(\Omega)\). If \(\mu \ge 0\) almost everywhere on \(\Omega\), then \(u \ge 0\) almost everywhere on \(\Omega\).
\end{lemma}

The proof of \Cref{lem:weak_maximum_principle} relies on a variant of Kato's inequality: if \(w \in L^1(\Omega)\), \(h \in L^1(\Omega, d_{\partial\Omega} \d{x})\) and \(\nu \in \cM(\partial\Omega)\) satisfy
\begin{equation}
\label{eq:identity_boundary}
- \int_\Omega w \Delta \zeta \d{x} = \int_\Omega h \zeta \d{x} + \int_{\partial\Omega} \frac{\partial \zeta}{\partial n} \d\nu \quad \text{for every \(\zeta \in C_0^\infty(\cl\Omega)\),}
\end{equation}
then
\begin{equation}
\label{eq:kato_inequality}
- \int_\Omega {} w^+ \Delta \zeta \d{x} \le \int_{\set{w \ge 0}} h \zeta \d{x} + \int_{\partial\Omega} \frac{\partial \zeta}{\partial n} \d\nu^+ \quad \text{for every \(\zeta \in C_0^\infty(\cl\Omega)\), \(\zeta \ge 0\) on \(\cl\Omega\)\,;}
\end{equation}
see \cite{MarcusVeron:1998}*{Lemma~1.5} or \cite{MarcusVeron:2014}*{Proposition~1.5.9}. Here \(\cM(\partial\Omega)\) denotes the vector space of finite Borel measures on \(\partial\Omega\) and
\[
C_0^\infty(\cl\Omega) = \set{\zeta \in C^\infty(\cl\Omega) : \text{\(\zeta = 0\) on \(\partial\Omega\)}}.
\]
When \(\nu = 0\), the integral identity \eqref{eq:identity_boundary} implicitly encodes the fact that \(w = 0\) on \(\partial\Omega\) in an average sense as test functions need not have compact support in \(\Omega\)\,; see \cite{Ponce:2016}*{Proposition~20.2} and also \cite{DiazRakotoson:2009} for related questions. To deduce the weak maximum principle it now suffices to take \(w = - u\), \(h = V u - \mu\) and \(\nu = 0\), and then \eqref{eq:kato_inequality} becomes
\[
- \int_\Omega {} (-u)^+ \Delta \zeta \d{x} \le 0 \quad \text{for every \(\zeta \in C_0^\infty(\cl\Omega)\), \(\zeta \ge 0\) on \(\cl\Omega\).}
\]

One last ingredient involved in the proof of \Cref{prop:pointwise_normal_derivative} is the following comparison principle:

\begin{lemma}
\label{lem:comparison_principle}
Let \(V_1, V_2 \in L_\loc^1(\Omega)\) be two nonnegative functions such that \(V_1 \le V_2\) almost everywhere on \(\Omega\), and let \(u_i \in L^1(\Omega) \cap L^1(\Omega; V_i d_{\partial\Omega} \d{x})\), with \(i \in \set{1, 2}\), be two nonnegative functions such that
\[
- \int_\Omega {} (u_2 - u_1) \Delta \zeta \d{x} + \int_\Omega {} (V_2 u_2 - V_1 u_1) \zeta \d{x} = 0 \quad \text{for every \(\zeta \in C_0^\infty(\cl\Omega)\).}
\]
Then \(u_2 \le u_1\) almost everywhere on \(\Omega\).
\end{lemma}

\Cref{lem:comparison_principle} can be deduced using Kato's inequality as above by taking \(w = u_2 - u_1\).

\begin{proof}[Proof of \Cref{prop:pointwise_normal_derivative}]
We assume that \(f\) is nonnegative; the general case follows by solving the Dirichlet problem with the positive and negative parts of \(f\), and then conclude using the linearity of the equation in \eqref{eq:dirichlet_problem} and uniqueness of solutions. Hence, by the weak maximum principle, \(u\) and \(u_k\) are nonnegative. Since \(u\) satisfies
\[
- \Delta u + V_k u = f - (V - V_k) u \quad \text{in the sense of distributions in \(\Omega\),}
\]
we have
\[
- \Delta (u_k - u) + V_k (u_k - u) = (V - V_k) u \quad \text{in the sense of distributions in \(\Omega\).}
\]
One deduces from \eqref{eq:W1p_estimate} applied to \(u_k - u\) that
\[
\norm{u_k - u}_{L^1(\Omega)} \le C \norm{(V - V_k) u}_{L^1(\Omega)}.
\]
By Lebesgue's dominated convergence theorem, the right-hand side of this inequality tends to \(0\) as \(k \to \infty\). 
Hence \(u_k \to u\) in \(L^1(\Omega)\). Since \((u_k)\) is non-increasing as a consequence of \Cref{lem:comparison_principle}, the convergence also holds everywhere on \(\Omega\). 
The triangle inequality and the absorption estimate \eqref{eq:absorption_estimate} applied to \(u_k - u\) imply that
\[
\norm{V_k u_k - V u}_{L^1(\Omega)} \le \norm{V_k (u_k - u)}_{L^1(\Omega)} + \norm{(V_k - V) u}_{L^1(\Omega)} \le 2 \norm{(V - V_k) u}_{L^1(\Omega)},
\]
and then \(V_k u_k \to V u\) in \(L^1(\Omega)\).

By comparison of normal derivatives, the sequence \((\partial u_k / \partial n)\) is non-increasing and nonnegative. In particular, it is uniformly bounded on \(\partial\Omega\) and converges in \(L^1(\partial\Omega)\) and everywhere on \(\partial\Omega\) to some nonnegative bounded measurable function \(g : \partial\Omega \to \R\). Let us show that
\[
g = \frac{\partial u}{\partial n} \quad \text{almost everywhere on \(\partial\Omega\),}
\]
where \(\partial u / \partial n\) is the distributional normal derivative of \(u\).
For this purpose, we recall that each \(u_k\) satisfies
\begin{equation}
\label{eq:uk_eq_boundary}
\int_\Omega \nabla u_k \cdot \nabla \psi \d{x} = \int_\Omega f \psi \d{x} - \int_\Omega V_k u_k \psi \d{x} - \int_{\partial\Omega} \frac{\partial u_k}{\partial n} \psi \d\sigma
\end{equation}
for every \(\psi \in C^\infty(\cl\Omega)\). By standard interpolation, which in this case follows from an integration by parts, one also has the estimate
\[
\norm{\nabla u_k}_{L^2(\Omega)}^2 \le \norm{u_k}_{L^\infty(\Omega)} \norm{\Delta u_k}_{L^1(\Omega)}\,;
\]
see \cite{Ponce:2016}*{Lemma~5.8}. We claim that the right-hand side of this inequality is bounded. Indeed, as \((u_k)\) is non-increasing, it is bounded from above by \(u_0\). On the other hand, we deduce from the triangle inequality and the absorption estimate \eqref{eq:absorption_estimate} that
\[
\norm{\Delta u_k}_{L^1(\Omega)} \le \norm{f}_{L^1(\Omega)} + \norm{V_k u_k}_{L^1(\Omega)} \le 2 \norm{f}_{L^1(\Omega)},
\]
which validates our claim.

Since \((\nabla u_k)\) is bounded in \(L^2(\Omega; \R^N)\) and \(u_{k} \to u\) in \(L^{1}(\Omega)\), we have
\[
\nabla u_k \weakto \nabla u \quad \text{weakly in \(L^{2}(\Omega; \R^N)\).}
\]
Taking the limit as \(k \to \infty\) in \eqref{eq:uk_eq_boundary}, we obtain
\[
\int_{\partial\Omega} \frac{\partial u}{\partial n} \psi \d\sigma = \int_{\partial\Omega} g \psi \d\sigma \quad \text{for every \(\psi \in C^\infty(\cl\Omega)\).}
\]
Hence \(\partial u / \partial n = g\) almost everywhere on \(\partial\Omega\). 
The estimate
\[
\norm{g}_{L^\infty(\partial\Omega)} \le C \norm{f}_{L^p(\Omega)}
\]
follows from \eqref{eq:Linf_boundary_estimate} since \(0 \le g \le \partial u_k / \partial n\) on \(\partial\Omega\).
\end{proof}

\section{Duality solution with measure data on the boundary}
\label{sec:boundary_value_problem}
\label{sec:L1_weak_derivative}

We investigate the boundary value problem \eqref{eq:boundary_value_problem} involving a finite Borel measure \(\nu\) on \(\partial\Omega\) by comparing two notions of solution based on different choices of test functions.

\begin{definition}
A function \(v \in L^1(\Omega)\) is a \emph{distributional} solution of \eqref{eq:boundary_value_problem} with datum \(\nu \in \cM(\partial\Omega)\) whenever \(V v \in L^1(\Omega; d_{\partial\Omega} \d{x})\) and
\[
\int_\Omega v \paren{- \Delta \zeta + V \zeta} \d{x} = \int_{\partial\Omega} \frac{\partial \zeta}{\partial n} \d\nu{}
\quad \text{for every \(\zeta \in C_0^\infty(\cl\Omega)\).}
\]
\end{definition}

The boundary value problem for this type of solutions has been studied by V\'eron and Yarur~\cite{VeronYarur:2012} with nonnegative potentials \(V \in L_\loc^\infty(\Omega)\). In particular, the authors prove that nonnegative measures for which \eqref{eq:boundary_value_problem} has a solution cannot charge \(\Sigma\)\,; see \cite{VeronYarur:2012}*{Theorem~4.4}. 
Their approach is based on the careful study of some capacity associated to the Poisson kernel of \(- \Delta\) on \(\Omega\). In our case, we rely instead on the concept of duality solution in the spirit of the work of Malusa and Orsina~\cite{MalusaOrsina:1996} that has its roots in the seminal paper of Littman, Stampacchia and Weinberger~\cite{LittmanStampacchiaWeinberger:1963}.

\begin{definition}
A function \(v \in L^1(\Omega)\) is a \emph{duality} solution of \eqref{eq:boundary_value_problem} with datum \(\nu \in \cM(\partial\Omega)\) whenever
\[
\int_\Omega v f \d{x} = \int_{\partial\Omega} \frac{\widehat{\partial \zeta_f}}{\partial n} \d\nu \quad \text{for every \(f \in L^\infty(\Omega)\),}
\]
where \(\zeta_f\) is the solution of \eqref{eq:dirichlet_problem} with datum \(f\).
\end{definition}

Existence of duality solutions is a straightforward consequence of the Riesz representation theorem:

\begin{proposition}
The boundary value problem \eqref{eq:boundary_value_problem} has a unique duality solution for every datum \(\nu \in \cM(\partial\Omega)\).
\end{proposition}

\begin{proof}
Let \(N < p < \infty\). It follows from \Cref{prop:pointwise_normal_derivative} that, for every \(f \in L^\infty(\Omega)\),
\[
\abs[\bigg]{\int_{\partial\Omega} \frac{\widehat{\partial \zeta_f}}{\partial n} \d\nu} \le \norm{\nu}_{\cM(\partial\Omega)} \norm[\bigg]{\frac{\widehat{\partial \zeta_f}}{\partial n}}_{L^\infty(\partial\Omega)} \le C \norm{\nu}_{\cM(\partial\Omega)} \norm{f}_{L^p(\Omega)},
\]
where
\[
\norm{\nu}_{\cM(\partial\Omega)} \defeq \abs{\nu}(\partial\Omega).
\]
Hence, the linear functional
\[
L^\infty(\Omega) \to \R : f \mapsto \int_{\partial\Omega} \frac{\widehat{\partial \zeta_f}}{\partial n} \d\nu
\]
is continuous on \(L^\infty(\Omega)\), endowed with the \(L^p\) norm. The Riesz representation theorem implies the existence of a unique \(v \in L^{p'}(\Omega)\) such that
\[
\int_\Omega v f \d{x} = \int_{\partial\Omega} \frac{\widehat{\partial \zeta_f}}{\partial n} \d\nu \quad \text{for every \(f \in L^\infty(\Omega)\),}
\]
where \(p' = \frac{p}{p - 1}\) is the conjugate exponent with respect to \(p\). Hence \(v\) is the unique duality solution of \eqref{eq:boundary_value_problem} involving \(\nu\).
\end{proof}

We now prove that distributional solutions are duality solutions:

\begin{proposition}
\label{prop:distributional_solutions_are_duality_solutions}
If \(v\) is a distributional solution of \eqref{eq:boundary_value_problem} with datum \(\nu \in \cM(\partial\Omega)\), then \(v\) is also a duality solution of \eqref{eq:boundary_value_problem} with datum \(\nu\).
\end{proposition}

For the proof of \Cref{prop:distributional_solutions_are_duality_solutions}, we need a couple of lemmas.
We begin with

\begin{lemma}
\label{lem:approximation_distributional_solutions}
Assume that \eqref{eq:boundary_value_problem} has a distributional solution \(v\) with datum \(\nu \in \cM(\partial\Omega)\) and let \(v_k\) be the distributional solution of
\begin{equation}
\label{eq:dirichlet_problem_approx}
\begin{dcases}
\begin{aligned}
- \Delta v_k + V_k v_k &= 0 && \text{in \(\Omega\),} \\
v_k &= \nu && \text{on \(\partial\Omega\).}
\end{aligned}
\end{dcases}
\end{equation}
Then \(v_k \to v\) in \(L^1(\Omega)\).
\end{lemma}

\begin{proof}[Proof of \Cref{lem:approximation_distributional_solutions}]
First notice that the function \(v - v_k\) satisfies
\[
\int_\Omega {} (v - v_k) (- \Delta \zeta + V_k \zeta) \d{x} = \int_\Omega (V_k - V) v \zeta \d{x} \quad \text{for every \(\zeta \in C_0^\infty(\cl\Omega)\).}
\]
Kato's inequality \eqref{eq:kato_inequality} applied to \(v - v_k\) and \(- (v - v_k)\) with \(\nu = 0\) implies that
\[
\int_\Omega {} \abs{v - v_k} (- \Delta \zeta + V_k \zeta) \d{x} \le \int_\Omega \sgn{(v - v_k)} (V_k - V) v  \zeta\d{x},
\]
and then
\[
\int_\Omega {} \abs{v - v_k} (-\Delta \zeta + V_k \zeta) \d{x} 
\le \int_\Omega {} \zeta (V - V_k) \abs{v} \d{x},
\]
for every \(\zeta \in C_0^\infty(\cl\Omega)\), \(\zeta \ge 0\) on \(\cl\Omega\). 
We take as test function the unique solution of
\begin{equation}
	\label{eq:test_function_Dirichlet}
\begin{dcases}
\begin{aligned}
- \Delta \theta &= 1 && \text{in \(\Omega\),} \\
\theta &= 0 && \text{on \(\partial\Omega\).}
\end{aligned}
\end{dcases}
\end{equation}
Observing that \(V_{k}\theta \ge 0\) and \(\theta \le \norm{\nabla \theta}_{L^\infty(\Omega)} d_{\partial\Omega}\) on \(\Omega\), we obtain the estimate
\begin{equation}
\label{eq:claim_L1_estimate}
\norm{v_k - v}_{L^1(\Omega)} \le \norm{\nabla \theta}_{L^\infty(\Omega)} \norm{(V - V_k) v}_{L^1(\Omega; d_{\partial\Omega} \d{x})}\,.
\end{equation}
Since \(0 \le V_k \le V\) and \(V v \in L^1(\Omega; d_{\partial\Omega} \d{x})\), Lebesgue's dominated convergence theorem implies that
\[
V_k v \to V v \quad \text{in \(L^1(\Omega; d_{\partial\Omega} \d{x})\).}
\]
The lemma then follows by letting \(k \to \infty\) in \eqref{eq:claim_L1_estimate}.
\end{proof}

\begin{lemma}
\label{prop:existence_bounded_potential}
Assume that \(V \in L^q(\Omega)\) for some \(q > N\). Then the boundary value problem \eqref{eq:boundary_value_problem} has a distributional solution for every \(\nu \in \cM(\partial\Omega)\). In particular, \(\Sigma = \emptyset\).
\end{lemma}

We recall that if \(v \in L^1(\Omega)\), \(f \in L^1(\Omega; d_{\partial\Omega} \d{x})\) and \(\nu \in \cM(\partial\Omega)\) satisfy
\[
- \int_\Omega v \Delta \zeta \d{x} = \int_\Omega f \zeta \d{x} + \int_{\partial\Omega} \frac{\partial \zeta}{\partial n} \d\nu \quad \text{for every \(\zeta \in C_0^\infty(\cl\Omega)\),}
\]
which is the weak formulation of \(v\) being a solution of
\[
\begin{dcases}
\begin{aligned}
- \Delta v &= f && \text{in \(\Omega\),} \\
v &= \nu && \text{on \(\partial\Omega\),}
\end{aligned}
\end{dcases}
\]
then
\begin{enumerate}[(i)]
\item for every \(1 \le p < \frac{N}{N - 1}\), we have \(v \in L^p(\Omega)\) and there exists a constant \(C > 0\) depending on \(p\) and \(\Omega\) such that
\begin{equation}
\label{eq:Lp_estimate}
\norm{v}_{L^p(\Omega)} \le C \paren[\big]{\norm{f}_{L^1(\Omega; d_{\partial\Omega} \d{x})} + \norm{\nu}_{\cM(\partial\Omega)}}\,;
\end{equation}
\item for every \(1 \le p < \frac{N}{N - 1}\) and every \(\omega \Subset \Omega\), we have \(\nabla v \in L^p(\omega; \R^N)\) and there exists a constant \(C' > 0\) depending on \(p\) and \(\omega\) such that
\begin{equation}
\label{eq:W1p_loc_estimate}
\norm{\nabla v}_{L^p(\omega; \R^N)} \le C' \paren[\big]{\norm{f}_{L^1(\Omega; d_{\partial\Omega} \d{x})} + \norm{\nu}_{\cM(\partial\Omega)}}.
\end{equation}
\end{enumerate}
We refer the reader to \cite{MarcusVeron:2014}*{Theorem~1.2.2} for a proof of these assertions.

\begin{proof}[Proof of \Cref{prop:existence_bounded_potential}]
Let \((g_k)\) be a sequence of smooth functions on \(\partial\Omega\) such that \(\norm{g_k}_{L^1(\partial\Omega)} \to \norm{\nu}_{\cM(\partial\Omega)}\) and \(g_k \weakstarto \nu\) \weakstartext{} in \(\cM(\partial\Omega)\), i.e.,
\[
\lim_{k \to \infty} \int_{\partial\Omega} \phi g_k \d{x} = \int_{\partial\Omega} \phi \d\nu \quad \text{for every \(\phi \in C(\partial\Omega)\).}
\]
Such a sequence can be obtained, for example, from a convolution of \(\nu\) with a sequence of mollifiers. Denote by \(v_k\) the distributional solution of \eqref{eq:boundary_value_problem} associated to \(g_k\). Given \(\omega \Subset \Omega\), we deduce from \eqref{eq:Lp_estimate} and \eqref{eq:W1p_loc_estimate} that
\[
\norm{v_k}_{W^{1, 1}(\omega)} \le C_1 \paren[\big]{\norm{V v_k}_{L^1(\Omega; d_{\partial\Omega} \d{x})} + \norm{g_k}_{L^1(\partial\Omega)}}
\]
for some constant \(C_1 > 0\) depending on \(\omega\). Taking a subsequence if necessary, we have
\begin{equation}
\label{eq:a.e._convergence}
v_k \to v \quad \text{almost everywhere on \(\Omega\).}
\end{equation}
On the other hand, since \(q' < \frac{N}{N - 1}\), we deduce from \eqref{eq:Lp_estimate} that
\[
\norm{v_k}_{L^{q'}(\Omega)} \le C_2 \paren[\big]{\norm{V v_k}_{L^1(\Omega; d_{\partial\Omega} \d{x})} + \norm{g_k}_{L^1(\partial\Omega)}}
\]
for some constant \(C_2 > 0\) depending on \(q\) and \(\Omega\). Hence \((v_k)\) is bounded in \(L^{q'}(\Omega)\). This, together with \eqref{eq:a.e._convergence}, implies that
\[
v_k \weakto v \quad \text{weakly in \(L^{q'}(\Omega)\).}
\]
Recalling that \(V \in L^q(\Omega)\) and taking the limit as \(k \to \infty\) in the equation
\[
\int_\Omega v_k (- \Delta \zeta + V \zeta) \d{x} = \int_\Omega \frac{\partial \zeta}{\partial n} g_k \d\sigma \quad \text{for every \(\zeta \in C_0^\infty(\cl\Omega)\),}
\]
we get the conclusion.
\end{proof}

We now turn to the

\begin{proof}[Proof of \Cref{prop:distributional_solutions_are_duality_solutions}]
We first assume that \(V\) is bounded. In this case, \(\zeta_f \in C_{0}^1(\cl\Omega)\) for every \(f \in L^\infty(\Omega)\). 
Since \(\zeta_f\) need not be smooth enough to be used as test function for \(v\), we approximate \(\zeta_f\) in \(C^1(\cl\Omega)\) by a sequence \((\zeta_{f_k})\) in \(C_0^\infty(\cl\Omega)\), where \((f_k)\) is a bounded sequence in \(L^\infty(\Omega)\) such that \(f_k \to f\) almost everywhere on \(\Omega\). For this purpose, we follow the construction given in \cite{OrsinaPonce:2019}: for each \(k\) we define the function \(g_k = \rho_k * g\), where \(g = f - V \zeta_f\) and \((\rho_k)\) is a sequence of mollifiers, and we denote by \(w_k \in C_0^\infty(\cl\Omega)\) the solution of
\[
\begin{dcases}
\begin{aligned}
- \Delta w_k &= g_k && \text{in \(\Omega\),} \\
w_k &= 0 && \text{on \(\partial\Omega\).}
\end{aligned}
\end{dcases}
\]
Observe that \(w_k = \zeta_{f_k}\) with \(f_k = g_k + V w_k\). Moreover, estimate \eqref{eq:C1_estimate} ensures that for some fixed \(N < p < \infty\),
\[
\norm{\zeta_{f_k} - \zeta_f}_{C^1(\cl\Omega)} \le C \norm{g_k - g}_{L^p(\Omega)}.
\]
Letting \(k \to \infty\) in this estimate, we have
\[
\zeta_{f_k} \to \zeta_f \quad \text{uniformly on \(\Omega\)} \quad \text{and} \quad \frac{\partial \zeta_{f_k}}{\partial n} \to \frac{\partial \zeta_f}{\partial n} \quad \text{uniformly on \(\partial\Omega\).}
\]
Since
\[
\int_\Omega v f_k \d{x} = \int_\Omega v (- \Delta \zeta_{f_k} + V \zeta_{f_k}) \d{x} = \int_{\partial\Omega} \frac{\partial \zeta_{f_k}}{\partial n} \d\nu,
\]
taking the limit as \(k \to \infty\), we obtain
\[
\int_\Omega  v f \d{x} = \int_{\partial\Omega} \frac{\partial \zeta_f}{\partial n} \d\nu.
\]

In the general case where \(V \in L_\loc^1(\Omega)\), let \(v_k\) be the solution of
\[
\begin{dcases}
\begin{aligned}
- \Delta v_k + V_k v_k &= 0 && \text{in \(\Omega\),} \\
v_k &= \nu && \text{on \(\partial\Omega\),}
\end{aligned}
\end{dcases}
\]
whose existence is ensured by \Cref{prop:existence_bounded_potential}. Given \(f \in L^\infty(\Omega)\), we denote by \(z_k \in C_0^1(\cl\Omega)\) the solution of
\[
\begin{dcases}
\begin{aligned}
- \Delta z_k + V_k z_k &= f && \text{in \(\Omega\),} \\
z_k &= 0 && \text{on \(\partial\Omega\).}
\end{aligned}
\end{dcases}
\]
It follows from the first part of the proof that
\[
\int_\Omega v_k f \d{x} = \int_{\partial\Omega} \frac{\partial z_k}{\partial n} \d\nu.
\]
\Cref{prop:pointwise_normal_derivative} implies that \(\paren{\partial z_k / \partial n}\) is uniformly bounded and converges pointwise to \(\widehat{\partial \zeta_f} / \partial n\) on \(\partial\Omega\). By \Cref{lem:approximation_distributional_solutions}, we obtain
\[
\int_\Omega v f \d{x} = \int_{\partial\Omega} \frac{\widehat{\partial \zeta_f}}{\partial n} \d\nu.\qedhere
\]
\end{proof}

\section{Pointwise normal derivative on the exceptional set \(\Sigma\)}
\label{sec:normal_derivative_on_sigma}

In this section, we characterize the exceptional boundary set \(\Sigma\) using the pointwise normal derivative with respect to the Schr\"odinger operator \(- \Delta + V\) introduced in \Cref{sec:pointwise_normal_derivative}. We prove that

\begin{proposition}
\label{prop:characterization_sigma}
For every \(a \in \partial\Omega\), we have \(a \in \Sigma\) if and only if
\[
\frac{\widehat{\partial \zeta_f}}{\partial n}(a) = 0 \quad \text{for every \(f \in L^\infty(\Omega)\).}
\]
\end{proposition}

As a fundamental property that is used in the proofs of \Cref{prop:characterization_sigma} and \Cref{thm:characterization_good_measures}, we first extend \Cref{lem:approximation_distributional_solutions} to the case where \eqref{eq:boundary_value_problem} need not have a distributional solution.

\begin{proposition}
\label{prop:duality_solutions_are_distributional_solutions}
Let \(\nu \in \cM(\partial\Omega)\) be a nonnegative measure and let \(v\) be the duality solution of \eqref{eq:boundary_value_problem} associated to \(\nu\). We have that
\begin{enumerate}[(i)]
\item if \(v_k\) is the distributional solution of \eqref{eq:dirichlet_problem_approx}, then \(v_k \to v\) in \(L^1(\Omega)\)\,;
\item there exists a nonnegative measure \(\lambda \in \cM(\partial\Omega)\) such that \(v\) is the distributional solution of \eqref{eq:boundary_value_problem} associated to \(\nu - \lambda\).
\end{enumerate}
\end{proposition}

We recall the following estimate whose proof is sketched for the convenience of the reader:

\begin{lemma}
\label{lem:L1_estimate}
If \(v\) is a distributional solution of \eqref{eq:boundary_value_problem} with \(\nu \in \cM(\partial\Omega)\), then
\[
\norm{v}_{L^1(\Omega)} + \norm{V v}_{L^1(\Omega; d_{\partial\Omega} \d{x})} \le C \norm{\nu}_{\cM(\partial\Omega)}
\]
for some constant \(C > 0\) depending on \(\Omega\).
\end{lemma}

\begin{proof}[Proof of \Cref{lem:L1_estimate}]
One deduces using Kato's inequality \eqref{eq:kato_inequality} with \(w = v\) and \(w = -v\) that
\[
\int_\Omega \abs{v} (-\Delta \zeta + V\zeta) \d{x} \le \norm[\bigg]{\frac{\partial \zeta}{\partial n}}_{L^\infty(\partial\Omega)} \norm{\nu}_{\cM(\partial\Omega)}
\]
for every \(\zeta \in C_0^\infty(\cl\Omega)\), \(\zeta \ge 0\) on \(\cl\Omega\). 
Take as test function the solution \(\theta\) of \eqref{eq:test_function_Dirichlet}.
As a consequence of the classical Hopf lemma, there exists \(C_1 > 0\) such that \(\theta \ge C_1 d_{\partial\Omega}\) on \(\Omega\). 
Therefore, by nonnegativity of \(V\),
\[
\norm{v}_{L^1(\Omega)} + C_1 \norm{V v}_{L^1(\Omega; d_{\partial\Omega} \d{x})} \le \norm[\bigg]{\frac{\partial \theta}{\partial n}}_{L^\infty(\partial\Omega)} \norm{\nu}_{\cM(\partial\Omega)}.
\qedhere
\]
\end{proof}

\begin{proof}[Proof of \Cref{prop:duality_solutions_are_distributional_solutions}]
By a straightforward counterpart of \Cref{lem:weak_maximum_principle,lem:comparison_principle} for distributional solutions of \eqref{eq:boundary_value_problem}, the sequence \((v_k)\) is nonnegative and non-increasing. Hence \((v_k)\) converges in \(L^1(\Omega)\) to some nonnegative function \(w\). Let \(f \in L^\infty(\Omega)\) and let \(u_k\) be the solution of \eqref{eq:dirichle_proposition_pointwise}.
\Cref{prop:distributional_solutions_are_duality_solutions} implies that
\[
\int_\Omega v_k f \d{x} = \int_{\partial\Omega} \frac{\partial u_k}{\partial n} \d\nu.
\]
By \Cref{prop:pointwise_normal_derivative}, the sequence \(\paren{\partial u_k / \partial n}\) is uniformly bounded and converges pointwise to \(\widehat{\partial \zeta_f} / \partial n\) on \(\partial\Omega\). Taking the limit as \(k \to \infty\) in the identity above, we deduce from Lebesgue's dominated convergence theorem that
\[
\int_\Omega w f \d{x} = \int_{\partial\Omega} \frac{\widehat{\partial \zeta_f}}{\partial n} \d\nu \quad \text{for every \(f \in L^\infty(\Omega)\).}
\]
We have thus proved that \(w\) is a duality solution of \eqref{eq:boundary_value_problem} involving \(\nu\). By uniqueness of duality solutions, we have \(v = w\).

For every \(k \ge 1\), we have
\[
0 \le V_k v_k \le V v_1 \quad \text{almost everywhere on \(\Omega\).}
\]
Since \(v_{1}\) is subharmonic, it is locally bounded on \(\Omega\)\,; see \cite{Willem:2013}*{Theorem~8.1.5}.
Then, Lebesgue's dominated convergence theorem implies that
\[
V_k v_k \to V v \quad \text{in \(L_\loc^1(\Omega)\).}
\]
Let \(\theta\) be the unique solution of \eqref{eq:test_function_Dirichlet}.
Since \(0 < \theta \le \norm{\nabla \theta}_{L^\infty(\Omega)} d_{\partial\Omega}\) on \(\Omega\), by \Cref{lem:L1_estimate} the sequence \((V_k v_k \theta)\) is bounded in \(L^1(\Omega)\). Therefore, taking a subsequence if necessary, we may assume that there exist nonnegative finite Borel measures \(\mu\) on \(\Omega\) and \(\tau\) on \(\partial\Omega\) such that, for every \(\psi \in C^\infty(\cl\Omega)\),
\begin{equation}
\label{eq:decomposition_limit}
\lim_{k \to \infty} {} \int_\Omega V_k v_k \theta \psi \d{x} = \int_\Omega \psi \d\mu + \int_{\partial\Omega} \psi \d\tau.
\end{equation}
On the other hand, for every \(\varphi \in C_c^\infty(\Omega)\),
\[
\lim_{k \to \infty} {} \int_{\Omega} V_k v_k \theta \varphi \d{x} = \int_\Omega V v \theta \varphi \d{x}.
\]
Hence \(\mu = V v \theta \d{x}\). Given \(\zeta \in C_0^\infty(\cl\Omega)\), we define \(\gamma = \zeta / \theta\) on \(\Omega\). Since \(\partial \theta / \partial n > 0\) on \(\partial\Omega\), the function \(\gamma\) extends continuously to \(\partial\Omega\), and
\[
\gamma = \frac{\partial \zeta}{\partial n} \frac{1}{\frac{\partial \theta}{\partial n}} \quad \text{on \(\partial\Omega\).}
\]
Taking \(\psi = \gamma\) in \eqref{eq:decomposition_limit}, we obtain
\[
\lim_{k \to \infty} {} \int_\Omega V_k v_k \zeta \d{x} = \int_\Omega V v \zeta \d{x} + \int_{\partial\Omega} \frac{\partial \zeta}{\partial n} \frac{1}{\frac{\partial \theta}{\partial n}} \d\tau \quad \text{for every \(\zeta \in C_0^\infty(\cl\Omega)\).}
\]
The result follows with \(\lambda = \dfrac{1}{\frac{\partial \theta}{\partial n}} \tau\).
\end{proof}

Another ingredient involved in the proof of \Cref{prop:characterization_sigma} is the inverse maximum principle for distributional solutions of \eqref{eq:boundary_value_problem}; see \cite{BrezisPonce:2005}*{Lemma~1}.

\begin{lemma}
\label{lem:inverse_maximum_principle}
Let \(\nu \in \cM(\partial\Omega)\) and let \(h \in L^1(\Omega; d_{\partial\Omega} \d{x})\). Assume that \(v \in L^1(\Omega)\) satisfies
\[
- \int_\Omega v \Delta \zeta \d{x} = \int_\Omega h \zeta \d{x} + \int_{\partial\Omega} \frac{\partial \zeta}{\partial n} \d\nu \quad \text{for every \(\zeta \in C_0^\infty(\cl\Omega)\).}
\]
If \(v \ge 0\) almost everywhere on \(\Omega\), then \(\nu \ge 0\) on \(\partial\Omega\).
\end{lemma}

Given \(a \in \partial\Omega\), we denote by \(P_a\) the duality solution of \eqref{eq:boundary_value_problem} associated to the Dirac measure \(\delta_a\), that is,
\begin{equation}
\label{eq:representation_formula}
\frac{\widehat{\partial \zeta_f}}{\partial n}(a) = \int_\Omega P_a f \d{x} \quad \text{for every \(f \in L^\infty(\Omega)\).}
\end{equation}
One deduces from the definition of duality solution using \(f = \chi_{\set{P_a < 0}}\) as test function that
\[
P_a \ge 0 \quad \text{almost everywhere on \(\Omega\).}
\]
We apply this simple observation in the

\begin{proof}[Proof of \Cref{prop:characterization_sigma}]
We first assume that \(a \in \Sigma\). By \Cref{prop:duality_solutions_are_distributional_solutions}, \(P_a\) is a distributional solution of \eqref{eq:boundary_value_problem} with datum \(\delta_a - \lambda\) for some nonnegative measure \(\lambda \in \cM(\partial\Omega)\). 
Since \(P_a \ge 0\) almost everywhere on \(\Omega\), we deduce from \Cref{lem:inverse_maximum_principle} that
\[
\delta_a \ge \lambda \ge 0 \quad \text{on \(\partial\Omega\).}
\]
Thus, \(\lambda = \alpha \delta_a\) for some \(0 \le \alpha \le 1\). If we had \(\alpha \neq 1\), then \(P_a / (1 - \alpha)\) would be a distributional solution of \eqref{eq:boundary_value_problem_dirac}, in contradiction with the assumption that \(a \in \Sigma\). 
Hence, \(\alpha = 1\) and
\[
\int_\Omega P_a (- \Delta \zeta + V \zeta) \d{x} = 0 \quad \text{for every \(\zeta \in C_0^\infty(\cl\Omega)\).}
\]
Taking \(\zeta = \theta\), where \(\theta\) satisfies \eqref{eq:test_function_Dirichlet}, we deduce that
\[
\int_\Omega P_a \d{x} = 0.
\]
Since \(P_a\) is nonnegative, we have \(P_a = 0\) almost everywhere on \(\Omega\). The representation formula \eqref{eq:representation_formula} then implies that
\[
\frac{\widehat{\partial \zeta_f}}{\partial n}(a) = 0 \quad \text{for every \(f \in L^\infty(\Omega)\).}
\]

For the converse, one deduces from the assumption on \(a\) and \eqref{eq:representation_formula} applied to \(f \equiv 1\) that
\[
\int_\Omega P_a \d{x} = 0.
\]
Hence \(P_a = 0\) almost everywhere on \(\Omega\), so that \(P_a\) cannot be a distributional solution of \eqref{eq:boundary_value_problem} involving \(\delta_a\). Since \(P_a\) is the only candidate for such a solution due to \Cref{prop:distributional_solutions_are_duality_solutions}, we conclude that \(a \in \Sigma\).
\end{proof}

\section{Proof of \Cref{thm:characterization_good_measures}}
\label{sec:proof_characterization_good_measures}

(\(\Leftarrow\)). Let \(v\) be the duality solution of \eqref{eq:boundary_value_problem} associated to \(\nu\). \Cref{prop:duality_solutions_are_distributional_solutions} implies the existence of a nonnegative measure \(\lambda \in \cM(\partial\Omega)\) such that \(v\) is a distributional solution of \eqref{eq:boundary_value_problem} involving \(\nu - \lambda\). We claim that \(\lambda(\partial\Omega \setminus \Sigma) = 0\). By \Cref{prop:distributional_solutions_are_duality_solutions}, \(v\) is also a duality solution of \eqref{eq:boundary_value_problem} with datum \(\nu - \lambda\). Hence
\[
\int_{\partial\Omega} \frac{\widehat{\partial \zeta_f}}{\partial n} \d\nu = \int_\Omega v f \d{x} = \int_{\partial\Omega} \frac{\widehat{\partial \zeta_f}}{\partial n} \d{(\nu - \lambda)} \quad \text{for every \(f \in L^\infty(\Omega)\),}
\]
which implies that
\[
\int_{\partial\Omega} \frac{\widehat{\partial \zeta_f}}{\partial n} \d\lambda = 0 \quad \text{for every \(f \in L^\infty(\Omega)\).}
\]
By \Cref{prop:characterization_sigma}, we have \(\widehat{\partial \zeta_1} / \partial n > 0\) on \(\partial\Omega \setminus \Sigma\). Since \(\lambda\) is nonnegative, we conclude that \(\lambda(\partial\Omega \setminus \Sigma) = 0\) as claimed. On the other hand, since \(v \ge 0\) almost everywhere on \(\Omega\), we deduce from \Cref{lem:inverse_maximum_principle} that
\[
\nu \ge \lambda \ge 0 \quad \text{on \(\partial\Omega\).}
\]
By assumption, \(\nu(\Sigma) = 0\). Hence \(\lambda(\Sigma) = 0\). We thus have
\[
\lambda(\partial\Omega) = \lambda(\partial\Omega \setminus \Sigma) + \lambda(\Sigma) = 0,
\]
that is, \(\lambda = 0\).

(\(\Rightarrow\)). Let \(v\) be the distributional solution of \eqref{eq:boundary_value_problem} associated to \(\nu\). \Cref{prop:distributional_solutions_are_duality_solutions} implies that \(v\) is also a duality solution of \eqref{eq:boundary_value_problem} involving the same datum. By \Cref{prop:characterization_sigma}, we have
\[
\int_\Omega v f \d{x} = \int_{\partial\Omega} \frac{\widehat{\partial \zeta_f}}{\partial n} \d\nu = \int_{\partial\Omega} \frac{\widehat{\partial \zeta_f}}{\partial n} \d\nu\lfloor_{\partial\Omega \setminus \Sigma} \quad \text{for every \(f \in L^\infty(\Omega)\),}
\]
so that \(v\) is also a duality solution of \eqref{eq:boundary_value_problem} with datum \(\nu\lfloor_{\partial\Omega \setminus \Sigma}\). The reverse implication in \Cref{thm:characterization_good_measures} implies that \eqref{eq:boundary_value_problem} associated to \(\nu\lfloor_{\partial\Omega \setminus \Sigma}\) has a unique distributional solution \(z\). But then, \Cref{prop:distributional_solutions_are_duality_solutions} ensures that \(z\) is also a duality solution of \eqref{eq:boundary_value_problem} with datum \(\nu\lfloor_{\partial\Omega \setminus \Sigma}\). Since duality solutions are unique, we have \(v = z\) almost everywhere on \(\Omega\). Thus, \(v\) is a distributional solution of \eqref{eq:boundary_value_problem} with both \(\nu\) and \(\nu\lfloor_{\partial\Omega \setminus \Sigma}\), which implies that \(\nu = \nu\lfloor_{\partial\Omega \setminus \Sigma}\), and then \(\nu(\Sigma) = 0\).

\begin{remark}
\label{rem:independence_from_Vk}
Using  \Cref{thm:characterization_good_measures} and \Cref{prop:characterization_sigma}, we can now explain why \(\widehat{\partial u} / \partial n\) does not depend upon the particular choice of approximating sequence \((V_k)\) in \Cref{prop:pointwise_normal_derivative}. 
Indeed, we have by \Cref{prop:characterization_sigma} that
\[{}
\frac{\widehat{\partial u}}{\partial n}(a)
= 0 \quad \text{for every \(a \in \Sigma\),}
\]
while \(\Sigma\) is independent of \((V_{k})\).{}
When \(a \in \partial\Omega \setminus \Sigma\), it follows from \Cref{thm:characterization_good_measures} that \(P_{a}\) is a distributional solution of \eqref{eq:boundary_value_problem_dirac}, whose definition does not involve \((V_{k})\).{}
Thus, by the representation formula \eqref{eq:representation_formula}, \(\widehat{\partial u} / \partial n\) is independent of \((V_{k})\) also on \(\partial\Omega \setminus \Sigma\).
\end{remark}

\section{Proof of \Cref{thm:universal}}
\label{sec:universal}

We deduce \Cref{thm:universal} as a consequence of

\begin{proposition}
\label{prop:universal}
Let \(u\) be the solution of \eqref{eq:dirichlet_problem} for some nonnegative datum in \(L^\infty(\Omega)\) such that \(\hat{u}\) has a classical normal derivative at \(a \in \partial\Omega\) which satisfies \eqref{eq:poisson_integral}. If \(v\) is another solution of \eqref{eq:dirichlet_problem} for some nonnegative datum in \(L^\infty(\Omega)\), and if \(v \le u\) almost everywhere on \(\Omega\), then \(\hat{v}\) also has a classical normal derivative at \(a\) which satisfies \eqref{eq:poisson_integral}.
\end{proposition}

We recall that whenever \(v\) satisfies \eqref{eq:dirichlet_problem} with datum \(h \in L^\infty(\Omega)\), one has
\[
\hat{v}(x) = \int_\Omega G(x, y) (h - V v)(x) \d{x} \quad \text{for every \(x \in \Omega\).}
\]
Using standard estimates on the Green function \(G\), by boundedness of \(h\) one shows that
\[
\lim_{t \downarrow 0} \int_\Omega \frac{G(a + t n, y)}{t} h(y) \d{y} = \int_\Omega K(a, y) h(y) \d{y},
\]
where \(n = n(a)\).
The main difficulty in the proof of \Cref{prop:universal} thus consists in getting that
\[
\lim_{t \downarrow 0} \int_\Omega \frac{G(a + t n, y)}{t} V v(y) \d{y} = \int_\Omega K(a, y) V v(y) \d{y}.
\]

\begin{proof}[Proof of \Cref{prop:universal}]
Let \((\varepsilon_k)\) be a non-increasing sequence of positive numbers converging to \(0\). We define on \(\Omega\)
\[
g_k(y) = \frac{G(a + \varepsilon_k n, y)}{\varepsilon_k}.
\]
On the one hand, we have \(g_k(y) \to K(a, y)\) for all \(y \in \Omega\). On the other hand, by assumption on \(u\),
\[
\int_\Omega g_k(y) V u(y) \d{y} \to \int_\Omega K(a, y) V u(y) \d{y}
\]
or, equivalently,
\[
\norm{g_k u}_{L^1(\Omega; V \d{x})} \to \norm{K(a, .) u}_{L^1(\Omega; V \d{x})}.
\]
We thus have pointwise convergence of \((g_k u)\) and also convergence of norms. Therefore, 
\[{}
g_k u \to K(a, .) u
\quad \text{in \(L^1(\Omega; V \d{x})\),}
\]
which is a special case of the Brezis--Lieb lemma~\cite{BrezisLieb:1983}; see \cite{Willem:2013}*{Proposition~4.2.6}. Since \(0 \le v \le u\), we deduce from Lebesgue's dominated convergence theorem that \(g_k v \to K(a, .) v\) in \(L^1(\Omega; V \d{x})\). The conclusion is now straightforward.
\end{proof}

For the proof of \Cref{thm:universal}, we also need the following

\begin{lemma}
\label{lem:integral_ineq}
Let \(u\) be the solution of \eqref{eq:dirichlet_problem} for some nonnegative datum \(f \in L^\infty(\Omega)\). Then, for every \(a \in \partial\Omega\), we have
\[
\limsup_{\varepsilon \downarrow 0} \frac{\hat{u}(a + \varepsilon n)}{\varepsilon}
\le \frac{\widehat{\partial u}}{\partial n}(a) 
\le \int_\Omega K(a, y)(f - V u)(y) \d{y}.
\]
\end{lemma}

\begin{proof}[Proof of \Cref{lem:integral_ineq}]
Let \(u_k\) satisfy \eqref{eq:dirichle_proposition_pointwise}. 
By comparison, we have \(\hat{u} \le u_k\) on \(\Omega\).{}
Hence,
\[{}
\limsup_{\varepsilon \downarrow 0} \frac{\hat{u}(a + \varepsilon n)}{\varepsilon}
\le \frac{\partial u_{k}}{\partial n}(a) 
\quad \text{for every \(a \in \partial\Omega\).}
\]
Then, as \(k \to \infty\) we deduce the first inequality in the statement.
 
Next, since \(\Delta u_k\) is bounded, for every \(a \in \partial\Omega\) we have
\[
\frac{\partial u_k}{\partial n}(a) = \int_\Omega K(a, y)(- \Delta u_k(y)) \d{y} = \int_\Omega K(a, y)(f - V_k u_k)(y) \d{y}.
\]
By \Cref{prop:pointwise_normal_derivative}, the sequence \((V_k u_k)\) converges almost everywhere to \(V u\). Fatou's lemma implies that, for every \(a \in \partial\Omega\),
\[
\int_\Omega K(a, y) V u(y) \d{y} 
\le \liminf_{k \to \infty} \int_\Omega K(a, y) V_k u_k(y) \d{y}.
\]
Observe that the right-hand side is finite since the sequence \((\partial u_k / \partial n)\) is bounded. Hence, for every \(a \in \partial\Omega\), we have
\[
\begin{aligned}
\frac{\widehat{\partial u}}{\partial n}(a) 
= \lim_{k \to \infty} \frac{\partial u_k}{\partial n}(a) 
&= \lim_{k \to \infty} \int_\Omega K(a, y)(f - V_k u_k)(y) \d{y} \\
&\le \int_\Omega K(a, y)(f - V u)(y) \d{y},
\end{aligned}
\]
which implies the second inequality in the statement.
\end{proof}

\begin{proof}[Proof of \Cref{thm:universal}]
The reverse implication (\(\Leftarrow\)) follows from \Cref{prop:universal} and the fact that \(u \le \zeta_{\norm{f}_{L^\infty(\Omega)}} = \norm{f}_{L^\infty(\Omega)} \zeta_1\) almost everywhere on \(\Omega\). We now prove the direct implication (\(\Rightarrow\)). 
One shows the existence of a constant \(C > 0\) such that, for every \(\varepsilon > 0\), the solution \(v_\varepsilon\) of the Dirichlet problem
\[
\begin{dcases}
\begin{aligned}
- \Delta v_\varepsilon + V v_\varepsilon &= \chi_{\set{u / C > \varepsilon}} && \text{in \(\Omega\),} \\
v_\varepsilon &= 0 && \text{on \(\partial\Omega\),}
\end{aligned}
\end{dcases}
\]
satisfies \(v_\varepsilon \le u / \varepsilon\) almost everywhere on \(\Omega\)\,; this is a consequence of Kato's inequality, as explained in the proof of Proposition~4.1 in \cite{OrsinaPonce:2018}. Let \((\varepsilon_j)\) be a non-increasing sequence of positive numbers converging to \(0\). 
Since \(\chi_{\set{u / C > \varepsilon_{j}}} \le 1\), by \Cref{prop:universal} the function \(\widehat{v_{\varepsilon_j}}\) has a normal derivative at \(a\) which satisfies \eqref{eq:poisson_integral}, which means that
\[{}
\frac{\partial \widehat{v_{\varepsilon_j}}}{\partial n}(a) = \int_\Omega K(a, y) (\chi_{\set{u / C > \varepsilon_j}} - V v_{\varepsilon_j}) (y) \d{y}.
\]
By the strong maximum principle for the Schrödinger operator with potential in \(L^{1}_{\loc}(\Omega)\), we have \(u > 0\) almost everywhere on \(\Omega\). Hence
\[{}
\chi_{\set{u / C > \varepsilon_{j}}} \to 1
\quad \text{almost everywhere on \(\Omega\).}
\]
The convergence thus holds in \(L^{1}(\Omega)\), which implies that
\[{}
v_{\varepsilon_j} \to \zeta_1{}
\quad \text{in \(L^{1}(\Omega)\).}
\]
As the sequence \((v_{\varepsilon_j})\) is nondecreasing, we deduce from Levi's monotone convergence theorem that
\[
\lim_{j \to \infty} \frac{\partial \widehat{v_{\varepsilon_j}}}{\partial n}(a) = \int_{\Omega} K(a, y) (1  - V\zeta_1(y)) \d{y}.
\]
Since \(\widehat{v_{\varepsilon_j}} \le \widehat{\zeta_1}\) on \(\Omega\), we also have, by classical comparison of limits,
\[
\lim_{j \to \infty} \frac{\partial \widehat{v_{\varepsilon_j}}}{\partial n}(a) \le \liminf_{j \to \infty} \frac{\widehat{\zeta_1}(a + \varepsilon_j n)}{\varepsilon_j}.
\]
\Cref{lem:integral_ineq} implies that
\[
\liminf_{k \to \infty} \frac{\widehat{\zeta_1}(a + \varepsilon_k n)}{\varepsilon_k} \le \limsup_{k \to \infty} \frac{\widehat{\zeta_1}(a + \varepsilon_k n)}{\varepsilon_k} \le \int_\Omega K(a, y) (1 - V \zeta_1(y)) \d{y}.
\]
Combining the inequalities above, we deduce that \(\partial \widehat{\zeta_1}(a) / \partial n\) exists and
\[
\frac{\partial \widehat{\zeta_1}}{\partial n}(a) = \lim_{k \to \infty} \frac{\partial \widehat{v_{\varepsilon_k}}}{\partial n}(a) = \int_{\Omega} K(a, y) (1  - V\zeta_1(y)) \d{y}.
\]
Hence, by definition, \(a \in \cN\).
\end{proof}

\section{Proof of \Cref{thm:hopf}}
\label{sec:proof_hopf}

We first prove a version of the Hopf lemma in terms of the pointwise normal derivative associated to the Schr\"odinger operator \(- \Delta + V\).{}
In this case, the answer does not involve the set \(\cN\).

\begin{proposition}
\label{prop:hopf_schrodinger}
Let \(u\) be the solution of \eqref{eq:dirichlet_problem} for some nonnegative datum \(f \in L^\infty(\Omega)\), \(f \not\equiv 0\). Then, for every \(a \in \partial\Omega\), we have
\[
\frac{\widehat{\partial u}}{\partial n}(a) > 0 \quad \text{if and only if} \quad a \not\in \Sigma.
\]
\end{proposition}

\begin{proof}
By \Cref{prop:characterization_sigma}, we only have to prove that \(\widehat{\partial u} / \partial n > 0\) on \(\partial\Omega \setminus \Sigma\). 
For this purpose, let \(a \in \partial\Omega \setminus \Sigma\).{}
In this case, \(P_a\) is both a duality and a distributional solution of \eqref{eq:boundary_value_problem} involving \(\delta_a\).{}
As a distributional solution, it satisfies the strong maximum principle for the Schr\"odinger operator with potential in \(L^{1}_{\loc}(\Omega)\). 
Hence, 
\begin{equation}
	\label{eq:Pa_positive}
P_a > 0
\quad \text{almost everywhere on \(\Omega\).}
\end{equation}
As a duality solution, \(P_a\) satisfies the representation formula \eqref{eq:representation_formula}.
Since \(f\) is nonzero, we then deduce from this formula and \eqref{eq:Pa_positive} that
\[
\frac{\widehat{\partial u}}{\partial n}(a) > 0.\qedhere
\]
\end{proof}

We now turn to the

\begin{proof}[Proof of \Cref{thm:hopf}]
By \Cref{thm:universal}, the classical normal derivative \(\partial \hat{u}/\partial n\) exists on \(\cN\) and satisfies \eqref{eq:poisson_integral}.	
A direct application of \Cref{lem:integral_ineq} gives, for every \(a \in \cN\),{}
\[
\frac{\partial \hat{u}}{\partial n}(a)
\le \frac{\widehat{\partial u}}{\partial n}(a) 
\le \int_\Omega K(a, y)(f - V u)(y) \d{y}.
\]
As the integral in the right-hand side equals \(\partial \hat{u}(a)/\partial n\), equality holds everywhere and we get
\[{}
\frac{\partial \hat{u}}{\partial n}
= \frac{\widehat{\partial u}}{\partial n}
\quad \text{on \(\cN\).}
\]
The theorem then follows from \Cref{prop:hopf_schrodinger}.
\end{proof}

We conclude this section with the following particular case of \Cref{thm:hopf}.

\begin{corollary}
\label{cor:hopf_lemma_supercritical}
Assume that \(V \in L^q(\Omega)\) for some \(q > N\). Then, for every solution \(u\) of \eqref{eq:dirichlet_problem} involving a nonnegative datum \(f \in L^\infty(\Omega)\), \(f \not\equiv 0\), the normal derivative of \(\hat{u}\) exists at every point \(a \in \partial\Omega\) and satisfies
\[
\frac{\partial \hat{u}}{\partial n}(a) > 0.
\]
\end{corollary}

\begin{proof}
We prove that \(\cN = \partial\Omega\) and \(\Sigma = \emptyset\). Let \(u\) be the solution of \eqref{eq:dirichlet_problem} involving some nonnegative datum \(f \in L^\infty(\Omega)\). 
Since \(u \in L^\infty(\Omega)\), we have \(\Delta u \in L^q(\Omega)\), and then \(u \in W^{2, q}(\Omega)\). 
The Morrey--Sobolev embedding theorem ensures that \(u \in C^1(\cl\Omega)\), which gives \(\cN = \partial\Omega\). 
That \(\Sigma = \emptyset\) follows from \Cref{prop:existence_bounded_potential}.
We then have the conclusion using \Cref{thm:hopf}.
\end{proof}

\section*{Acknowledgements}

The first author (A.C.P.) was supported by the Fonds de la Recherche scientifique (F.R.S.--FNRS) under research grant  J.0020.18.

\end{document}